\documentclass[12pt]{amsart}
\usepackage{amssymb,amsmath,MnSymbol,xcolor}
\usepackage{comment}
\usepackage{mathtools}
\usepackage[margin=1.25in]{geometry} 
\usepackage[bookmarks, bookmarksdepth=2, colorlinks=true, linkcolor=blue, citecolor=blue, urlcolor=blue]{hyperref}
\usepackage{enumitem}

\let\ul\underline

\newtheorem{theorem}{Theorem}[section]
\newtheorem{lemma}[theorem]{Lemma}
\newtheorem{proposition}[theorem]{Proposition}
\newtheorem{corollary}[theorem]{Corollary}
\theoremstyle{definition}
\newtheorem{definition}[theorem]{Definition}
\newtheorem{claim}[theorem]{Claim}

\def\multiset#1#2{\ensuremath{\left(\kern-.3em\left(\genfrac{}{}{0pt}{}{#1}{#2}\right)\kern-.3em\right)}}

\newcommand{\F}{\mathbf{F}}

\newcommand{\Q}{\mathbf{Q}}
\newcommand{\C}{\mathbf{C}}

\newcommand{\A}{\mathbf{A}}
\renewcommand{\P}{\mathbf{P}}

\renewcommand{\sp}{\textnormal{span}}
\newcommand{\codim}{\textnormal{codim}}

\newcommand{\rk}{\textnormal{rk}}

\newcommand{\brk}{\textnormal{Brk}}

\let\ul\underline

\author{Amichai Lampert}
\address{Department of Mathematics, University of Michigan, Ann Arbor, MI}
\email{\href{mailto:amichai@umich.edu}{amichai@umich.edu}}
\thanks{The author was supported by NSF grant DMS-2402041}

\title{Density of solutions for systems of forms}
\date{}

\begin{document}

\begin{abstract}
    Let $K$ be a field of characteristic zero over which every diagonal form in sufficiently many variables admits a nontrivial solution. For example, $K$ may be a totally imaginary number field or a finite extension of a $p$-adic field. Suppose $f_1,\ldots,f_s$ are forms of degree $d$ over $K.$ Bik, Draisma and Snowden recently proved that there exists a constant $B = B(d,s,K)$ such that the rational solutions to the system of equations $f_1=\ldots=f_s = 0$ are Zariski dense, as long as the Birch rank of $f_1,\ldots,f_s$ is greater than $B.$ We establish an effective bound for this constant, improving vastly on the astronomical bound coming from their proof. Our result has applications for surjectivity of polynomial maps and for the Hardy-Littlewood circle method.
\end{abstract}

\maketitle

\section{Introduction}

We are interested in studying solutions to systems of forms (i.e. homogeneous polynomials) over a field $K.$ An especially simple class of forms is the diagonal ones. 

\begin{definition}
    The quantity  $\phi_d = \phi_d(K)$ is defined to be the minimal integer such that the following holds: Whenever $n > \phi_d$ and $a\in K^n$ there is some $0\neq x\in K^n$ satisfying $a_1x_1^d+\ldots+a_nx_n^d = 0.$ If there is no such threshold, then $\phi_d = \infty.$ 
\end{definition}

For example, $\phi_1 =1$ over any field and $\phi_2(\Q) = \infty,$ since $x_1^2+\ldots+x_n^2$ admits no rational solutions (besides $x = 0$). From now on we assume that $\phi_d$ is finite for all $d.$ Fields with this property are called \emph{Brauer fields} (see subsection \ref{Brauer-examples} for examples). Consider a collection of forms (i.e. homogeneous polynomials) $ f_1,\ldots,f_s \in K[x_1,\ldots,x_n]$ of common degree $d.$ 

\subsection{Existence of solutions} In a landmark 1945 paper \cite{Brauer}, Brauer proved:

\begin{theorem}[Brauer]\label{thm:Brauer}
     There exists a constant $V_{d,s} = V(d,s,\phi_2,\ldots,\phi_d)$ such that if $n > V_{d,s}$ then there is some $0\neq x\in K^n$ with $f_1(x)=\ldots=f_s(x)=0.$
\end{theorem}

The proof of theorem \ref{thm:Brauer} proceeded via an elegant diagonalization argument, but resulted in an extremely large value for $V_{d,s}.$ It was long thought impossible to obtain a reasonable bound by similar methods. Almost forty years later, Leep and Schmidt \cite{LS} succeeded in that task, via carefully setting up an efficient inductive scheme. Their bound was further improved by Wooley \cite{Wooley}.

\begin{theorem}[Wooley]\label{thm:Wooley}
    Theorem \ref{thm:Brauer} holds with 
    \[
    V_{d,s} = 2s^{2^{d-1}} \phi_d^{2^{d-2}} \prod_{i=2}^{d-1} (\phi_i+1)^{2^{i-2}}.
    \]
\end{theorem}

\subsection{Density of solutions}

Suppose now that the field $K$ is infinite. Bik, Draisma and Snowden recently \cite{BDS} proved a density version of theorem \ref{thm:Brauer} in this setting. To state it requires the following definition.

\begin{definition}
    The \emph{Birch rank} of $f$ is
    \[
    \brk(f) = \min(n,\codim_{\A^n} (x:\nabla f(x)= 0)).
    \]
    The minimum only plays a role when $f$ is linear. More generally, the collective Birch rank of $f_1,\ldots,f_s$ is 
    \[
    \brk(f_1,\ldots,f_s) = \min \{ \brk(a_1f_1+\ldots+a_sf_s): 0\neq a \in \overline{K}^s \}.\footnote{This is not the standard definition, but it differs from it by at most $s-1$.}
    \]
\end{definition}

Their result is as follows.\footnote{Their result is phrased in terms of the strength of $f_1,\ldots,f_s,$ but is equivalent to theorem \ref{thm:BDS} as stated. This is because Birch rank and strength are both bounded in terms of each other, see \cite{BDS-strength}.}

\begin{theorem}(Bik, Draisma and Snowden)\label{thm:BDS}
     There exists a constant $B_{d,s} = B(d,s,\varphi_2,\ldots,\varphi_d)$ such that if $\brk(f_1,\ldots,f_s) > B_{d,s}$ then the $K$-points of $Z(f_1,\ldots,f_s)$ are Zariski dense.
\end{theorem}

The hypothesis of large Birch rank is common in applications of the Hardy-Littlewood circle method, but the above result is novel in that its proof methods are entirely \emph{algebro-geometric} rather than \emph{analytic}. This is readily apparent from the variety of fields for which the result applies (see subsection \ref{Brauer-examples} for examples). Theorem \ref{thm:BDS} has an interesting consequence regarding surjectivity of certain polynomial maps. Suppose that $f_1,\ldots,f_s\in K[x_1,\ldots,x_n]$ are polynomials of degree $d,$ not necessarily homogeneous. Let $F_1,\ldots,F_s$ be the degree $d$ parts of $f_1,\ldots,f_s.$

\begin{corollary}\label{cor:surj}
    If $\brk(F_1,\ldots,F_s) > \max(B_{d,s}+2, 2s-2)$ then the map $\ul{f}:K^n\to K^s$ is surjective. 
\end{corollary}

As in theorem \ref{thm:Brauer}, the constant $B_{d,s}$ that can be extracted from the proof of theorem \ref{thm:BDS} is astronomical.  Our main result is a reasonable bound for this constant, of a similar quality to that of theorem \ref{thm:Wooley}.  

Assume henceforth that the characteristic of $K$ is either zero or greater than $d.$

\begin{theorem}\label{thm:main}
Theorem \ref{thm:BDS} (and hence corollary \ref{cor:surj}) holds for $B_{d,s} = s^{2^{d-1}}C_d,$ where
\[
C_d = 2^{(2\varphi)^{d-2}} \prod_{k=2}^d (\phi_k+11)^{2^{d-2}(\varphi^{d-k-1}+2^{k-d})+4}
\]
and $\varphi = \frac{1+\sqrt{5}}{2}\approx 1.62$ is the golden ratio.    
\end{theorem}

Bounds for $\phi_d$ were obtained by Skinner \cite{Skinner-phi} when $K$ is a finite extension of $\Q_p.$ These bounds yield the following corollary to our theorem. 

\begin{corollary}\label{cor:p-adic} 
If $K$ is a finite extension of $\Q_p$ then theorem \ref{thm:BDS} (and hence corollary \ref{cor:surj}) holds with
\[
B_{d,s} = s^{2^{d-1}} (4d)^{(2\varphi)^{d-1}+2^d+8d}  \cdot 2^{(2\varphi)^{d-2}}. 
\]
\end{corollary}

If $\ul{f}$ has coefficients in a number field and $\brk(\ul{f})$ is larger than the above bound, then corollary \ref{cor:p-adic} guarantees the existence of a \emph{smooth} solution to $\ul{f} = 0$ at every finite place. This property is of significant interest for the Hardy-Littlewood circle method, see \cite{Birch,Skinner-HLS}.

\subsection{Examples}\label{Brauer-examples}
    The following are some examples of Brauer fields (see \cite{BDS} for details and more examples):
\begin{itemize}
    \item Totally imaginary number fields \cite{Peck},
    \item $p$-adic fields $\Q_p$,
    \item function fields $\F_q(t_1,\ldots,t_m)$ or $\C(t_1,\dots,t_m)$,
    \item Laurent series $\F_q(\!(t_1,\ldots,t_m)\!)$ or $\C(\!(t_1,\ldots,t_m)\!)$
    \item and any finite extension of the above.
\end{itemize}

\subsection{Structure of the paper}
The next section begins by proving corollaries \ref{cor:surj} and \ref{cor:p-adic}, after which we state our three main technical propositions. We end section 2 by briefly discussing some key ideas in our proof. In section 3, theorem \ref{thm:main} is deduced from the three main propositions. Section 4 collects various geometric results required for the proofs of the main propositions, which are then carried out in sections 5-7.  

\subsection{Acknowledgements} The author thanks Tamar Ziegler for introducing him to Schmidt's result \cite{Schmidt} and related work. Many thanks are due to Andrew Snowden for useful discussions about the proof of theorem \ref{thm:BDS} and to Trevor Wooley for graciously sharing his knowledge of the state of the art. Last but certainly not least, I thank Noa for her unflagging encouragement and support throughout this project.

\section{Proof of corollaries and statement of propositions}

\subsection{Proof of corollaries \ref{cor:surj} and  \ref{cor:p-adic}}

We may assume that $d\ge 2$ since the statements are trivial for linear forms. We begin with some simple properties of Birch rank. Let $h_1,\ldots,h_s$ be forms of degree $d.$

\begin{claim}\label{brk-basic}\mbox{}
\begin{enumerate}
    \item Subadditivity:
    $\brk(h_1+h_2) \le \brk(h_1)+\brk(h_2).$
    \item If $\brk(h_1,\ldots,h_s) \ge 2s-1$ then $Z(h_1,\ldots,h_s)$ is a complete intersection of codimension $s.$
\end{enumerate}
    
\end{claim}

\begin{proof}
    Subadditivity follows from the containment 
    \[
    (x:\nabla h_1(x) = 0) \cap (x:\nabla h_2(x) = 0) \subset (x:\nabla (h_1+h_2)(x) = 0)
    \]
    and the fact that codimension is subadditive on intersections. 

    For the second property, suppose, to get a contradiction, that $Z(g_1,\ldots,g_s)$ has a component of codimension $<s.$ Any smooth point in that component must be contained in the locus $S = (x: \nabla g_1(x),\ldots,\nabla g_s(x) \textnormal{ are linearly dependent}).$ On the other hand, this locus may be written as $S = \bigcup_{a\in \P^{s-1}} S_a,$ with $ S_a =  (x: \nabla (a_1g_1+\ldots+a_sg_s)(x) = 0).$ By assumption, each $S_a$ satisfies $\codim S_a \ge 2s-1$ and so by a union bound $\codim S \ge s,$ contradiction.
\end{proof}

\begin{proof}[Proof of corollary \ref{cor:surj}]
Since the hypothesis is invariant under substitutions of the type $f_i\mapsto f_i-c_i,$ it's enough to prove that $Z(f_1,\ldots,f_s)$ contains a $K$-point. Let $g_i = y^d f_i(x_1/y,\ldots,x_n/y)$ be the homogenized forms corresponding to the $f_i.$ We claim that 
\begin{equation}\label{eq:brk-homog}
    \brk(g_1,\ldots,g_s) \ge \brk(F_1,\ldots,F_s) -2 > B_{d,s}
\end{equation}

Assuming this inequality, we complete the proof. On the subspace $y=0,$ the $g_i$ restrict to $F_i$ and these define a complete intersection of codimension $s$ by claim \ref{brk-basic}. Therefore $y,g_1,\ldots,g_s$ cut out a complete intersection of codimension $s+1$ so in particular $y$ does not vanish identically on $Z = Z(g_1,\ldots,g_s).$ By inequality \eqref{eq:brk-homog}, theorem \ref{thm:BDS} implies that $Z$ contains a $K$-point $(x_1,\ldots,x_n,y)$ with $y\neq 0.$ The $K$-point $(x_1/y,\ldots,x_n/y,1)$ is therefore contained in $Z(f_1,\ldots,f_s)$ as desired.    

To prove inequality \eqref{eq:brk-homog}, note that for any non-trivial linear combination we have 
\[
a_1g_1+\ldots+a_sg_s = a_1F_1+\ldots+a_sF_s + yh,
\]
where $h$ is some form of degree $d-1.$ The reducible form $yh$ is easily seen to have $\brk(yh) \le 2,$ by Krull's height theorem. It follows from claim \ref{brk-basic} that 
\[
\brk(a_1F_1+\ldots+a_sF_s) \le \brk(a_1g_1+\ldots+a_sg_s) +2.
\]
\end{proof}

Next we prove corollary \ref{cor:p-adic}.

\begin{proof}[Proof of corollary \ref{cor:p-adic}]
    Skinner proved \cite{Skinner-phi}  $\phi_k \le 8k^2.$ This implies 
    \begin{align*}
         \prod_{k=2}^d (\phi_k+11)^{2^{d-2}(2^{k-d}+\varphi^{d-k-1})+4} &\le           (4d)^{2^{d-1} \sum_{k=2}^d (2^{k-d}+\varphi^{d-k-1}) +8d} \\
         &\le (4d)^{(2\varphi)^{d-1} + 2^d +8d}. 
    \end{align*}
    The result follows by plugging this into the bound of theorem \ref{thm:main}.
\end{proof}

\subsection{Three propositions}
The inductive proof will proceed via a more general version of theorem \ref{thm:main}. Consider a system of forms $\ul{f} = (f_{i,j})_{i\in [d],j\in [s_i]},$ where each $f_{i,j}$ has degree $i.$ The definition of Birch rank generalizes as follows.

\begin{definition}
    The Birch rank of $\ul{f}$ is 
    \[
    \brk(\ul{f}) = \min_{i\in [d]} \brk(f_{i,1},\ldots,f_{i,s_i}).
    \]
\end{definition}

Theorem \ref{thm:main} is a special case of the following result.

\begin{theorem}\label{thm:main-detailed}
    Let $s = \max_{i\in [d]} s_i^{2^{i-d}}.$ If 
    \[
    \brk(\ul{f}) > s^{2^{d-1}} 2^{(2\varphi)^{d-2}} \prod_{k=2}^d (\phi_k+11)^{2^{d-2}(\varphi^{d-k-1}+2^{k-d})+4}
    \]
    then the $K$-points of $Z(\ul{f})$ are Zariski dense.
\end{theorem}

The proof of theorem \ref{thm:main-detailed} relies on three main propositions. Before we can state them, we need to set up some definitions.

\begin{definition} 
Define $B = B(s_d,\ldots,s_1)$ to be the minimal integer such that for any collection of forms $\ul{f} = (f_{i,j})_{i\in [d],j\in [s_i]}$ with $\brk(\ul{f}) > B,$ the $K$-points of $Z(\ul{f})$ are Zariski dense.  We sometimes write $B_d(s_d,\ldots,s_1)$ if the degree is otherwise unclear. For example, $B_1(s) = B(s) = s.$ 
\end{definition}

There is a special case of this which will be particularly important for us. A \emph{non-degenerate system of diagonal forms} $\ul{f} = (f_1,\ldots,f_{d+1})$ is a collection where:
\begin{itemize}
    \item Each $f_i\in K[x_1,\ldots,x_n]$ is a diagonal form with nonzero coefficients,
    \item The degree of $f_i$ equals $i$ for $1\le i\le d$ and $\deg(f_{d+1}) > d.$ 
\end{itemize}

\begin{definition}
    Define $\beta_d$ to be the minimal integer such that the following holds. Whenever $n > \beta_d$ and $\ul{f} = (f_1,\ldots,f_{d+1})$ is a non-degenerate system of diagonal forms then there exists $x\in K^n$ with $f_1(x) = \ldots = f_d(x) = 0$ and $f_{d+1}(x)\neq 0.$
\end{definition}

Our first proposition is a bound for $\beta_d,$ which may be of independent interest.

\begin{proposition}[Bound for non-degenerate diagonal systems]\label{prop:beta-bound}
    There exists a sequence of even integers $n_d$ such that for $d\ge 1$ we have
    \[
    \beta_d+1 \le n_d \le  2^{1+\varphi^{d-1}}\prod_{k=2}^{d} (\phi_k+11) ^{\varphi^{d-k}}.
    \]
\end{proposition}

To state our next propositions, recall that the Taylor expansion of a form $f$ of degree $d$ is given by
\[
f(x_1+\ldots+x_m) = \sum_{|e| = d} f^e(x_1,\ldots,x_m),
\]
where each $f^e$ is multi-homogeneous of multi-degree $e.$ Write $\multiset{m}{d} = \binom{m+d-1}{d}$ for the number of terms in the above sum. For a collection of forms $\ul{f} = (f_{i,j})_{i\in [d],j\in [s_i]},$ write $T_m(\ul{f}) = (f_{i,j}^e)_{i\in [d],j\in [s_i],|e|=i}$ for the forms on $V^m$ given by the Taylor expansion. These are also the coefficients of the polynomials obtained by restricting $\ul{f}$ to $\sp(x_1,\ldots,x_m).$ We abbreviate $Z_m(\ul{f}) := Z(T_m(\ul{f})).$

Assume $s_d>0$ and write $\ul{f} = \ul{g} \sqcup \{f\},$ where $f = f_{d,s_d}.$ Recall the following useful definition from \cite{BDS}.

\begin{definition}
    Subspaces $U_1,\ldots,U_m$ are $f$-orthogonal if
    \[
    f(u_1+\ldots+u_m) = f(u_1)+\ldots+f(u_m)
    \]
    for any tuple of vectors $u_i\in U_i.$ Equivalently, $f^e(u_1,\ldots,u_m) = 0$ for any exponent $e$ with multiple nonzero entries.  
\end{definition}

We consider the algebraic set 
\[
D_m = ((w_1,w_2,w_3,v_1,\ldots,v_{m-3})\in Z_m(\ul{g}): \sp(w_1,w_2,w_3),v_1,\ldots,v_{m-3} \textnormal{ are } f\textnormal{-orthogonal}).
\]

\begin{definition}
   Let $\delta = \delta_m(s_d,\ldots,s_1)$ be the minimal number such that whenever $\brk(\ul{f}) > \delta$ the $K$-points of $D_m$ are Zariski dense. 
\end{definition}

Our second main proposition bounds $B(s_d,\ldots,s_1)$ in terms of $\delta_m(s_d,\ldots,s_1).$

\begin{proposition}[Diagonalization implies density]\label{prop:diag-to-good}
    For $d\ge 2$ define
    \[
    m = m_d = (\phi_d+9)\left\lceil \frac{\beta_{d-2}+1}{2} \right\rceil.
    \]
    Then 
    $$B(s_d,\ldots,s_1) \le \max(\delta_m(s_d,\ldots,s_1), 2 \sum_{i=1}^d s_i \multiset{m}{i}+d'-1).$$
\end{proposition}

Conversely, $\delta_m(s_d,\ldots,s_1)$ can be bounded in terms of $B(s_d^*,\ldots,s_1^*),$ where $(s_d^*,\ldots,s_1^*)$ is a lexicographically smaller sequence. This is the content of our third main proposition.

\begin{proposition}[Efficient diagonalization]\label{prop:diag}
    For any $m$ we have 
    \[
    \delta_m(s_d,\ldots,s_1) \le 2\sum_{i=1}^d s_im^i+B(s^*_d,\ldots,s^*_1),
    \]
    where $s^*_d = s_d-1$ and $s^*_i = s_i+ms_{i+1}+\ldots+s_dm^{d-i}$ for $1\le i\le d-1.$ 
\end{proposition}

Taken together, these three propositions allow us to prove theorem \ref{thm:main-detailed} by induction on the sequence $(s_d,\ldots,s_1).$ This analysis is carried out in the next section.

\subsection{Proof idea} Theorem \ref{thm:main-detailed} is achieved by synthesizing the techniques used in the proofs of theorems \ref{thm:Wooley} and \ref{thm:BDS}. This is much easier said than done and requires substantially altering the proof of theorem  \ref{thm:BDS}, overcoming significant technical hurdles along the way. Rather than dwelling on these,  we emphasize a key change in perspective compared to the proof of theorem \ref{thm:BDS} -- this is the use of Birch rank rather than strength as our fundamental quantity. 

Although Birch rank and strength are bounded in terms of each other (see \cite{BDS-strength,BL, BDLZ}), simple examples (such as $x_1^2+\ldots + x_n^2$ over $\Q(i)$) show that they may differ by a constant factor. If the right hand side of proposition \ref{prop:diag} were so much as doubled, then our inductive argument would no longer result in polynomial bounds for theorem \ref{thm:main-detailed}. Thus, a key reason that the proof techniques of theorem \ref{thm:Wooley} may be adapted to the setting of Birch rank is the \emph{linear behavior} it exhibits -- in the sense that proposition \ref{prop:diag} holds. We believe that this linear behavior will find further interesting applications.
  
Another important property of Birch rank (which does  also hold for strength) is that it is preserved under Taylor expansion, in the sense of lemma \ref{lem:brk-taylor}. This is the reason we need to assume that the characteristic is either zero or greater than $d.$ Since our argument proceeds by restricting forms to subspaces with desirable properties, this property plays a major role in our proofs.

\section{Proof of theorem \ref{thm:main-detailed}}

In this section we deduce the main theorem from  propositions \ref{prop:diag-to-good}, \ref{prop:diag} and \ref{prop:beta-bound}. We begin with some simple observations.

\begin{lemma}\label{lem:monotone+linears}
    The function $B(s_d,\ldots,s_1)$ is monotonically increasing. Furthermore, 
    \begin{align*}
        B(s_d,\ldots,s_1) &\le 2s_1+B(s_d,\ldots,s_2,0) \\
        \delta_m(s_d,\ldots,s_1) &\le 2s_1+\delta_m(s_d,\ldots,s_2,0).
    \end{align*}
\end{lemma}

\begin{proof}
    The first claim is that 
    \[
    B(s_d,\ldots,s_i+1,\ldots,s_1) \ge B(s_d,\ldots,s_i,\ldots,s_1).
    \]
    Denote the expression on the left-hand side by $B$ and suppose $\ul{f} = (f_{i,j})_{i\in [d],j\in [s_i]}$ is a collection of forms with $\brk(\ul{f}) > B.$ We need to show that $K$-points are dense in $Z(\ul{f}).$ Let $W\cong V$ and let $g$ be a smooth $K$-form of degree $i$ on $W.$ Consider the collection of forms $\ul{f}' = (f_{i,j}\circ \pi_1)_{i\in [d],j\in [s_i]} \cup\{g\circ \pi_2\}$ on $V\times W,$ where $\pi_1:V\times W\to V $ and $\pi_2:V\times W\to W$ are the projections. It's easily verified that $\brk(\ul{f}') = \brk(\ul{f}) > B$ so that $K$-points are dense in $Z(\ul{f}').$ Since $Z(\ul{f}') = Z(\ul{f})\times Z(g),$ we deduce that the $K$-points of $Z(\ul{f})$ are dense, as desired. 

     Now, supposing $\brk(\ul{f}) > B(s_d,\ldots,s_2,0)+2s_1,$ we claim that the $K$-points of $Z(\ul{f})$ are dense. Let $U\subset V$ be the $K$-subspace of codimension $s_1$ where the linear forms in $\ul{f}$ vanish and consider $\ul{g} = (g_{i,j})_{2\le i\le d, j\in [s_i]},$ where $g_{i,j} = f\restriction_U.$ By lemma \ref{lem:brk-subspace}, we have $\brk(\ul{g}) > B(s_d,\ldots,s_2,0),$ so that the $K$-points of $Z(\ul{g}) = Z(\ul{f})$ are dense, as claimed. The proof of the corresponding inequality for $\delta_m$ is identical.
\end{proof}
 
Iterating propositions \ref{prop:diag-to-good} and \ref{prop:diag} allows us to bound $B_d$ in terms of $B_{d-1}$ and $m_d.$

\begin{lemma}[Degree-lowering]\label{lem:deg-lower}
    Let $m = m_d$ be as in proposition \ref{prop:diag-to-good}. Then 
    \[
    B_d(s_d,\ldots,s_1) \le 
    2 \sum_{i=1}^{d} s_i (ms_d)^i  +  B_{d-1}(r_{d-1},\ldots,r_1),
    \]
    where $r_i = s_i + (ms_d)s_{i+1}+\ldots+(ms_d)^{d-i} s_d.$
\end{lemma}

\begin{proof}

We begin by noting that 
    \[
    2\sum_{i=1}^d s_i m^i \ge 2\sum_{i=1}^d s_i\multiset{m}{i} + d'-1.
    \]
Combining propositions \ref{prop:diag-to-good} and \ref{prop:diag} and using the above observation gives
\[
    B_d(s_d,s_{d-1},\ldots,s_2,s_1) \le  2\sum_{i=1}^d s_im^i+B(s^*_d,\ldots,s^*_2,s^*_1),
    \]
    where $s^*_d = s_d-1$ and $s^*_i = \sum_{k\ge i} s_im^{k-i}$ for $1\le i\le d-1.$ Iterating this, let $s_i^{(j)}$ be the number of degree $i$ forms after $j$ iterations. Thus, 
    \begin{align*}
         s_i^{(0)} &= s_i\ ,\ s_d^{(j)} = s_d-j \\ s_i^{(j+1)} &= \sum_{k\ge i} s_k^{(j)} m^{k-i}, \quad 1\le i\le d-1.
    \end{align*}
    After $s_d$ iterations we end up with
    
\begin{equation}\label{eq:main-prop-iterated}
     B_d(s_d,s_{d-1},\ldots,s_2,s_1) \le 2\sum_{j=0}^{s_d-1}\sum_{i=1}^d s_i^{(j)}m^i + B_{d-1}(s_{d-1}^{(s_d)},\ldots,s_2^{(s_d)},s_1^{(s_d)}).
\end{equation}

Note that
\begin{equation}\label{eq:s_i-iterated}
    s_i^{(j)} \le \sum_{\ell \ge i} s_\ell(jm)^{\ell-i}.
\end{equation}
This is proved by downward induction on $i,$ then upward on $j.$ The base cases $i =d$ and $j = 0$ are clear. Supposing inequality \eqref{eq:s_i-iterated} holds for $i' >i$ and for $s_i^{(j)},$ we have 
\begin{align*}
    s_i^{(j+1)} &= \sum_{k\ge i} s_k^{(j)} m^{k-i} \le \sum_{k\ge i} m^{k-i} \sum_{\ell \ge k} s_\ell(jm)^{\ell-k} \\
    &= \sum_{\ell \ge i} s_\ell m^{\ell-i} \sum_{k=i}^\ell  j^{\ell-k} \le \sum_{\ell \ge i} s_\ell m^{\ell-i} (j+1)^{\ell-i},
\end{align*}
which proves the inductive step.

Therefore,
\begin{multline}\label{eq:s_i-sum-iterated}
        \sum_{j=0}^{s_d-1}\sum_{i=1}^d s_i^{(j)}m^i \le \sum_{j=0}^{s_d-1}\sum_{i=1}^d m^i \sum_{\ell \ge i} s_\ell(jm)^{\ell-i}  = \sum_{\ell = 1}^d s_\ell m^\ell \sum_{j=0}^{s_d-1} \sum_{i=1}^\ell j^{\ell-i}\\
        \le \sum_{\ell = 1}^d s_\ell m^\ell \sum_{j=0}^{s_d-1}  (j+1)^{\ell-1} \le \sum_{\ell = 1}^d s_\ell (ms_d)^\ell 
\end{multline} 

The proof is completed by plugging inequalities \eqref{eq:s_i-iterated} and \eqref{eq:s_i-sum-iterated} into \eqref{eq:main-prop-iterated} and using the monotonicity property from lemma \ref{lem:monotone+linears}.
\end{proof}

We now specialize this bound.

\begin{lemma}\label{lem:powers-bound}
    We have 
    \[
    B(s,s^2) \le 2s^2(m_2+1)^2
    \]
    and for every $d\ge 3,$
    \[
    B_d(s,s^2,\ldots,s^{2^{d-1}}) \le d s^{2^{d-1}} (m_2+1)^2(m_3+1)^3\prod_{i=4}^d (m_i+1)^{2^{i-2}}. 
    \]
\end{lemma}

\begin{proof}
    The proof is by induction on $d.$ When proving the claim for $d$ we write $m = m_d.$ For $d=2$ we combine propositions \ref{prop:diag-to-good}, \ref{prop:diag} and lemma \ref{lem:monotone+linears} to get
    \begin{align*}
        B(s,0) &\le  2sm^2 +B(s-1,ms) \le 2s(m^2+m) +B(s-1,0) \\
        &\le \ldots \le 2s^2 (m^2+m).
    \end{align*}
    By lemma \ref{lem:monotone+linears},  
    \[
    B(s,s^2) \le 2s^2+B(s,0) \le 2s^2(m+1)^2.
    \]
    For $d = 3,$ we apply lemmas \ref{lem:deg-lower} and \ref{lem:monotone+linears} to get
     \begin{align*}
        B_3(s,s^2,s^4) &\le  2(m^2s^4 + m^3s^4+s^4) + B_2(s^2+ms^2,ms^3+m^2s^3)  \\
    &\le 2s^4(m+1)^3 + B_2(s^2(m+1),s^4(m+1)^2) \\
    &\le 2s^4(m_3+1)^3 + 2s^4(m_3+1)^2(m_2+1)^2 \le 3s^4 (m_2+1)^2(m_3+1)^3,
    \end{align*}
    where the penultimate inequality used the bound for $d=2.$ 

    For $d \ge 4,$  plugging in $s_i = s^{2^{d-i}}$ to lemma \ref{lem:deg-lower} yields
    \begin{align*}
        B_d(s,s^2,\ldots,s^{2^{d-2}},0) &\le  2\sum_{i=2}^d s^{2^{d-i}} (ms)^i +  B_{d-1}(r_{d-1},\ldots,r_2,r_1)  \\
    &\le 2s^{2^{d-1}}\cdot (m+1)^d + B_{d-1}(r_{d-1},\ldots,r_2,r_1) 
    \end{align*}
    For $1\le i\le d-1$ note that
    \begin{align*}
        r_i &=  s^{2^{d-i}} + (s m) s^{2^{d-i-1}}+\ldots+(sm)^{d-i} s \\
        &\le s^{2^{d-i}}(m+1)^{d-i} \le [s^2(m+1)]^{2^{d-1-i}}.
    \end{align*}
    By inductive hypothesis and lemma \ref{lem:monotone+linears},
    \begin{align*}
        B_d(s,s^2,\ldots,s^{2^{d-1}}) &\le s^{2^{d-1}}\cdot [2(m_d+1)^d+2] \\
        &+(s^2(m_d+1))^{2^{d-2}} (d-1) (m_2+1)^2(m_3+1)^3 \prod_{i=4}^{d-1} (m_i+1)^{2^{i-2}} \\
        &\le d s^{2^{d-1}} (m_2+1)^2(m_3+1)^3\prod_{i=4}^d (m_i+1)^{2^{i-2}},
    \end{align*}
    which completes the induction step.
\end{proof}

We are now ready to prove our main result.

\begin{proof}[Proof of theorem \ref{thm:main-detailed}]
By lemma \ref{lem:monotone+linears}, it's enough to prove the bound for $B(s,s^2,\ldots,s^{2^{d-1}}).$ Thanks to lemma \ref{lem:powers-bound}, all that remains is to bound $(m_2+1)^2(m_3+1)^3\prod_{i=4}^d (m_i+1)^{2^{i-2}}.$ This will be accomplished using proposition \ref{prop:beta-bound}. 

By definition, $m_2+1 \le \phi_2+11.$ By proposition \ref{prop:beta-bound} we have $m_3+1 \le 2(\phi_3+11)$ and in general 
\begin{align*}
    \prod_{i=4}^d (m_i+1)^{2^{i-2}} &\le 2^{\sum_{i=4}^d 2^{i-2}\varphi^{i-3}} \prod_{i=4}^d (\phi_i+11)^{2^{i-2}} \prod_{k=2}^{i-2} (\phi_k+11)^{2^{i-2}\varphi^{i-2-k}} \\
    &\le 2^{(2\varphi)^{d-2}}  \prod_{k=2}^{d} (\phi_k+11)^{2^{k-2} +\sum_{i=k+2}^d 2^{i-2}\varphi^{i-2-k}} \\
    &\le 2^{(2\varphi)^{d-2}} \prod_{k=2}^d (\phi_k+11)^{2^{d-2}(2^{k-d}+\varphi^{d-k-1})}.
\end{align*}
    
\end{proof}

\section{Some geometric results}

We now collect some preliminary results that will be needed in the course of the proofs of propositions \ref{prop:diag-to-good} and \ref{prop:diag}.

\subsection{Taylor expansion and Birch rank} We begin by recording several elementary identities for the Taylor expansion.

\begin{lemma}\label{lem:Taylor-identities} 
    The terms of the Taylor expansion satisfy the following formulas:
    \begin{enumerate}
        \item $f^{(1,\ldots,1)}(x_1,\ldots,x_d) = f^{(1,\ldots,1)}(x_{\sigma(1)},\ldots,x_{\sigma(d)})$ for any permutation $\sigma\in S_d,$ 
        \item $f^{(1,d-1)}(h,x) = \nabla f(x) \cdot h,$
        \item $e_1!\ldots e_m!f^e(x_1,\ldots,x_m) = f^{(1,\ldots,1)} (x_1,\ldots,x_1,\ldots,x_m,\ldots,x_m),$ where each $x_i$ appears $e_i$ times and
        \item   for any scalars $\lambda_1,\ldots,\lambda_m,$   
        \[
            \frac{\partial f^e }{\partial x_k}(\lambda_1 x,\ldots, \lambda_m x) = \binom{d-1}{e_1,\ldots,e_k-1,\ldots,e_m} \cdot \lambda_1^{e_1}\ldots\lambda_k^{e_k-1}\ldots\lambda_m^{e_m} \nabla f(x).
        \]
        
    \end{enumerate}
\end{lemma}

\begin{proof}
    For the first property, note that for any permutation $\sigma$ we have 
    $$f(x_1+\ldots+x_d) = f(x_{\sigma(1)}+\ldots+x_{\sigma(d)}).$$
    Comparing the multi-homogeneous terms of multi-degree $(1,\ldots,1)$ on both sides proves the claim.

    For the second formula, first note that $f^{(0,d)}(h,x) = f(x)$ follows from setting $h=0.$ Therefore, $f(x+th) = f(x) + t f^{(1,d-1)}(h,x) +O(t^2),$ proving the claim. 

    To prove the third formula, compute the coefficient of $\prod_{i=1}^m \prod_{j=1}^{e_i} t_{i,j}$ in the expression $f\left (\sum_{i=1}^m \sum_{j=1}^{e_i}  t_{i,j} x_i \right)$ in two different ways. On the one hand, using the $d$-term Taylor expansion, this equals $f^{(1,\ldots,1)} (x_1,\ldots,x_1,\ldots,x_m,\ldots,x_m).$ On the other hand, using the $m$-term Taylor expansion, this equals $e_1!\ldots e_m! f^e(x_1,\ldots,x_m).$

     For the fourth identity, both sides are multi-homogeneous of the same multi-degree, so we may assume without loss of generality that $\lambda_1 = \ldots = \lambda_m = 1$ and that $k=1.$ Using formulas (1-3) and the multi-linearity of $f^{(1,\ldots,1)},$ we get
    \begin{align*}
        e_1!\ldots e_m! f^e(x+th,x,\ldots,x) &= f^{(1,\ldots,1)}(x+th,\ldots,x+th,x,\ldots,x) \\
        &= e_1!\ldots e_m! f^e(x,\ldots,x) + t e_1\cdot f^{(1,\ldots,1)}(h,x,\ldots,x) + O(t^2) \\
        &= e_1!\ldots e_m! f^e(x,\ldots,x) + t e_1(d-1)! f^{(1,d-1)} (h,x) + O(t^2) \\
        &= e_1!\ldots e_m! f^e(x,\ldots,x) + t e_1(d-1)! \nabla f(x)\cdot h + O(t^2),
    \end{align*}
    which proves the claim.
\end{proof}

This last formula has a very useful consequence.

\begin{lemma}\label{lem:brk-taylor}
    For any $m\ge 1,$ $\brk(T_m(\ul{f})) \ge \brk(\ul{f}).$ 
\end{lemma}

\begin{proof}
    The statement is about each degree separately, so we may assume without loss of generality that $\ul{f}=(f_1,\ldots,f_s)$ is composed entirely of forms of degree $d.$ For $d=1$ the claim is immediate so assume also that $d>1.$ Let $g = \sum_{j\in [s], |e| = d} a_{j,e}f^e_j $ be a nontrivial linear combination and $S(g) = (\ul{x}: \nabla g(\ul{x}) = 0).$  For $\lambda = (\lambda_1,\ldots,\lambda_m) \neq 0,$ consider the diagonal embedding $V\cong V_\lambda \subset V^m$ given by the image of  $x\mapsto (\lambda_1 x,\ldots,\lambda_m x).$ 
    
    The intersection $S(g)\cap V_\lambda$ is nonempty because it contains $0,$ and we have
    \begin{equation}\label{eq:diag-dim}
        \codim_{V_\lambda} (S(g)\cap V_\lambda) \le \codim_{V^m} S(g) = \brk(g). 
    \end{equation}
    By assumption, there is some $e_0$ for which $a_{j,e_0}$ are not all zero. Let $k\in [m]$ be such that $(e_0)_k>0.$ By lemma \ref{lem:Taylor-identities}, 
    \[
    \frac{\partial g}{\partial x_k} (\lambda_1 x,\ldots,\lambda_m x) = \sum_{j\in [s]} \nabla f_j(x) \sum_{|e|=d, e_k>0} a_{j,e} \binom{d-1}{e_1,\ldots,e_k-1,\ldots,e_m} \cdot \lambda_1^{e_1}\ldots\lambda_k^{e_k-1}\ldots\lambda_m^{e_m} .
    \]
    Choose $\lambda$ such that the coefficient of some $\nabla f_j(x)$ above is nonzero and let 
    \[
    f = \sum_{j\in [s]} f_j \sum_{|e|=d, e_k>0} a_{j,e} \binom{d-1}{e_1,\ldots,e_k-1,\ldots,e_m} \cdot \lambda_1^{e_1}\ldots\lambda_k^{e_k-1}\ldots\lambda_m^{e_m}
    \]
    be the corresponding nontrivial linear combination of the $f_j.$ We have 
    \[
    S(f)\supset (x\in V: (\lambda_1 x,\ldots,\lambda_m x) \in S(g)) , 
    \]
    so that
    \[
    \brk(\ul{f}) \le \brk(f) = \codim_V S(f) \le \codim_{V_\lambda} (S(g)\cap V_\lambda).
    \]
    Combining this with inequality \eqref{eq:diag-dim} gives $\brk(\ul{f}) \le \brk(g)$ as desired.
\end{proof}

\subsection{Birch rank, restriction to subspaces and singularities}

We have already mentioned that our arguments will often rely on restricting forms to subspaces. An issue that arises is that the Birch rank of the restricted form may be smaller than that of the original form. 

Consider, for example, the quadratic form $f = x^2-y^2$ and the hyperplane $U  = \{(x,x):x\in K\}.$ We have $\brk(f) = 2$ but $\brk(f\restriction_U) = 0,$ so the Birch rank decreases by $2.$ Fortunately, this is the largest possible drop in Birch rank upon restriction to a hyperplane -- as we will soon see. This is a consequence of the following result of Ananyan and Hochster \cite[Theorem 2.4]{AH}.

\begin{lemma}\label{lem:mat-diff-deg}
    Let $M$ be a $h\times n$ matrix such that for each $i\in [h]$ the $i$-th row consists of forms of degree $d_i$ in $K[x_1,\ldots,x_n]$ and the $d_i$ are mutually distinct integers. If for each $i\in[h]$ we have
    \[
    \codim_{\A^n} (x:M_{i,j}(x) = 0\ \forall j\in [n]) \ge c
    \]
    then
    \[
    \codim_{\A^n} (x:\rk(M(x)) < h) \ge c-h+1.
    \] 
\end{lemma}

We can now easily bound the decrease in Birch rank of the restriction of a form to a subspace.

\begin{lemma}\label{lem:brk-subspace}
    Let $\ul{f}$ be a collection of forms of degree two or higher and $U\subset K^n$ a subspace of codimension $c.$
    Then $\brk(\ul{f}\restriction_U) \ge \brk(\ul{f}) -2c.$
\end{lemma}

\begin{proof}
    It is enough to prove this in the case of a single form $f$ and $c=1.$  In this case, $U = (x\in K^n: a_1x_1+\ldots+a_n x_n = 0)$ for some $0\neq a\in K^n.$ We have
    \begin{align*}
        \brk(f\restriction_U) &= \codim_U (x\in U: \nabla f(x)\in \sp(a)) \\
        &\ge \codim_{\A^n} (x\in \A^n: (\nabla f(x),a) \textnormal{ are linearly dependent}) - 1 \ge \brk(f) -2,
    \end{align*}
    where the final inequality follows from lemma \ref{lem:mat-diff-deg}.
\end{proof}

Lemma \ref{lem:mat-diff-deg} also implies a close connection between $\brk(\ul{f})$ and the codimension of 
$$S(\ul{f}) = (x\in \A^n: (\nabla f_{i,j}(x))_{i\in [d],j\in [s_i]} \textnormal{ are linearly dependent}).$$

\begin{lemma}\label{lem:brk-sing}
    We have 
    $$\codim_{\A^n} S(\ul{f}) \ge \brk(\ul{f}) -\sum_{i=1}^d s_i-d'+2,$$
    where $d' = |\{i:s_i\neq 0\}|.$
\end{lemma}

\begin{proof}
Let $s = \sum_{i=1}^d s_i.$ Write $S = \cup_{a\in \P^{s-1}} S_a,$ where 
\[
S_a =  \left(x: \sum_{i=1}^d\sum_{j=1}^{s_i} a_{i,j} \nabla f_{i,j}(x) = 0 \right).
\]
It's enough to show that for any fixed $a\in \P^{s-1}$ we have $\codim_{\A^n} S_a \ge \brk(\ul{f})-d'+1.$ Fixing $a,$ let $I = \{i\in [d]: (a_{i,j})_{j\in [s_i]} \neq (0,\ldots,0)\}.$ For each $i\in I,$ let $f_i = \sum_{j\in [s_i]} a_i f_{i,j}.$ We now have the inclusion 
\[
S_a \subset (x\in \A^n: (\nabla f_i(x))_{i\in I} \textnormal{ are linearly dependent}).
\]

Applying lemma \ref{lem:mat-diff-deg} yields
\begin{align*}
    \codim_{A^n} S_a &\ge \codim_{\A^n} (x\in \A^n: (\nabla f_i(x))_{i\in I} \textnormal{ are linearly dependent}) \\
    &\ge \brk(\ul{f}) - |I|+1 \ge \brk(\ul{f}) - d'+1. 
\end{align*}  
\end{proof}

There is an important criterion relating singularities to irreducibility.

\begin{lemma}\label{lem:R1-seq}
   Let $f_1,\ldots,f_s$ be a sequence of forms of positive (possibly different) degrees such that $ Z = Z(f_1,\ldots,f_s)$ is a complete intersection of codimension $s.$ If 
   \[
   \codim_Z (x\in Z: (\nabla f_1(x),\ldots,\nabla f_s(x) \textnormal{ are linearly dependent}) >1, 
   \]
   then the ideal $(f_1,\ldots,f_s)$ is prime. In particular, $Z(f_1,\ldots,f_s)$ is irreducible. 
\end{lemma}

\begin{proof}
    This is a well known fact that follows from \cite[Theorem 18.15b]{Eisenbud}. Note that the hypothesis of that theorem is satisfied by \cite[Proposition 18.13]{Eisenbud} and that $Z(f_1,\ldots,f_s)$ is connected since it's a cone. Therefore, if its coordinate ring is a product of domains then it must in fact be a domain.
\end{proof}

As a consequence of the results of this section, we have:

\begin{lemma}\label{lem:brk-irr}
    If $\brk(\ul{f}) > 2\sum_{i=1}^d s_i \multiset{m}{i}+d'-1$ and $\ul{g} \subset T_m(\ul{f})$ consists of $t$ equations then $Z(\ul{g})$ is an irreducible complete intersection of codimension $t.$
\end{lemma}

\begin{proof}
    It follows from lemmas \ref{lem:brk-taylor} and \ref{lem:brk-sing} that
    \begin{equation}\label{eq:Sing-g}
        \codim_{V^m} S(\ul{g}) \ge \brk(\ul{f})-\sum_{i=1}^d s_i \multiset{m}{i}-d'+2 > \sum_{i=1}^d s_i \multiset{m}{i}+1.
    \end{equation}
    The right hand side is larger than $t,$ so in particular the zero locus $Z = Z(\ul{g})$ is a complete intersection of codimension $t.$ Inequality \eqref{eq:Sing-g} also implies that we have $\codim_Z Z\cap S(\ul{g}) > 1,$ so $Z$ is irreducible by lemma \ref{lem:R1-seq}.
\end{proof}

\section{Proof of proposition \ref{prop:diag}}

The proof will rely on the following geometric result.

\begin{lemma}\label{lem:diag-on-subspace-geo}
    Let $m>1$ and suppose  $\brk(\ul{f}) > 2\sum_{i=1}^d s_i m^i+c.$ Let $h$ be a polynomial which doesn't vanish identically on $D_m.$ Then for a dense open set of $\ul{x}\in D_{m-1},$  
    \begin{enumerate}
        \item The polynomial $h(\ul{x},\cdot)$ doesn't vanish identically on the fiber $D_{m}(\ul{x})$ and 
        \item We have $\brk \left( (f_{i,j}^e(\ul{x},\cdot))_{e_m>0} \right) > c.$ 
    \end{enumerate}
\end{lemma}

\begin{proof}
    Note that 
    \[
    2\sum_{i=1}^d s_i m^i \ge 2\sum_{i=1}^d s_i\multiset{m}{i} + d'-1,
    \]
    so by lemma \ref{lem:brk-irr} both $D_{m-1}$ and $D_m$ are irreducible. It's therefore enough to show that each of the conditions holds on a nonempty open subset of $D_{m-1}.$
    
    For the first condition, irreducibility of $D_m$ implies that the proper subset $D' = (h = 0)\cap D_m$ has dimension $\dim D' < \dim D_m.$ Therefore, for a dense open subset of $\ul{x}\in D_{m-1}$ the fiber satisfies $\dim D'(\ul{x}) <\dim D_m(\ul{x})$ as desired. 
    
    For the second condition, again by irreducibility it's enough to show this for each fixed $e_m = k>0.$ Let 
    \[
    S = (\ul{x}\in V^{m-1}: \brk \left((f_{i,j}^e(\ul{x},\cdot))_{e_m = k} \right) \le c)
    \]
    It's certainly enough to prove that
    \begin{equation}\label{eq:codim-S}
        \codim_{V^{m-1}} S > \sum_{i=1}^d s_i \multiset{m}{i} > \codim_{V^{m-1}} D_{m-1}.
    \end{equation}
    
    To prove this, consider some nontrivial linear combination $g = \sum_{e_m = k} a_{i,j}^e f_{i,j}^e(x_1,\ldots,x_m)$ and let
    $$S_a = (\ul{x}\in V^{m-1}: \brk(g(\ul{x},\cdot) \le c))\ ,\ Y_a = \left( \ul{x}\in V^m : \frac{\partial g}{\partial x_m} (\ul{x}) = 0 \right).$$
    For $ 0 \neq \lambda\in K^m,$ we have
    \[
    \codim_{V^m} Y_a \ge \codim_V (x\in V:(\lambda_1x,\ldots,\lambda_m x) \in Y_a). 
    \]
    By lemma \ref{lem:Taylor-identities}, the expression on the right hand side is
    \[
    \codim_V \left( x\in V:  \sum_{i\ge k,j\in [s_i]} \nabla f_{i,j}(x) \sum_{e_m = k} a_{i,j}^e \cdot \binom{i}{e_1,\ldots,e_{m-1},k-1} \cdot \lambda_1^{e_1}\ldots \lambda_{m-1}^{e_{m-1}} \lambda_m^{k-1}  = 0 \right).
    \]
    Choosing $\lambda$ such that the coefficients of  $\nabla f_{i,j}(x)$ above are not all zero and applying lemma \ref{lem:brk-sing}, we obtain  
    \[
    \codim_{V^m} Y_a \ge \codim_V S(\ul{f}) > \brk(\ul{f}) -\sum_{i=1}^d s_i-d' \ge 2\sum_{i=1}^d s_i\multiset{m}{i} +c .
    \]
    This gives  
    \[
    \codim_{V^{m-1}} S_a = \codim_{V^{m-1}} (\ul{x}: \codim_V Y_a(\ul{x}) \le c) > 2\sum_{i=1}^d s_i \multiset{m}{i}.
    \]
    Taking the union over $a\in \P^{\sigma-1}$ for
    $\sigma = \sum_{i=k}^d s_i \multiset{m}{i-k}$ yields 
    \[
    \codim_{V^{m-1}} S > \min_{a\in \P^{\sigma-1}} \codim_{V^{m-1}} S_a - \sigma >  \sum_{i=1}^d s_i \multiset{m}{i},  
    \]
    proving inequality \eqref{eq:codim-S}.
\end{proof}

We can now prove the proposition.

\begin{proof}[Proof of proposition  \ref{prop:diag}]
    Recall that our goal is to prove that
     \[
    \delta_m(s_d,\ldots,s_2,s_1) \le 2\sum_{i=1}^d s_im^i+B(s^*_d,\ldots,s^*_1),
    \]
    where $s^*_d = s_d-1$ and $s^*_i = s_i+ms_{i+1}+\ldots+s_dm^{d-i}$ for $1\le i\le d-1.$ The proof is by induction on $m.$ In the base case $m=1$ we have $D_1 = Z(\ul{g})$ so by lemma \ref{lem:monotone+linears}, 
    \[
    \delta_1(s_d,\ldots,s_1) = B_d(s_d-1,s_{d-1},\ldots,s_1) \le B_d(s_d^*,\ldots,s_1^*).
    \] 
    
    Suppose now that the claim holds for $m-1$ and we wish to prove it for $m.$ We assume that $\brk(\ul{f}) > 2\sum_{i=2}^d s_im^i+B(s^*_d,\ldots,s^*_1)$ and that $h$ is a polynomial not vanishing identically on $D_m,$ and need to show that $(h\neq 0)\cap D_m$ contains a $K$-point. Let $U\subset D_{m-1}$ be the dense open subset guaranteed by lemma \ref{lem:diag-on-subspace-geo}, with $c = B(s^*_d,\ldots,s^*_1).$ 
     
    By the inductive hypothesis, we may choose a $K$-point $\ul{x}\in U.$ The fiber $D_m(\ul{x})$ is then defined by at most $s_i^*$-equations of degree $i,$ and these equations are a subset of $(f_{i,j}^e(\ul{x},\cdot))_{e_m>0}$ so their Birch rank is greater than $B(s^*_d,\ldots,s^*_1).$ It follows that we may choose a $K$-point $y\in D_m(\ul{x}) \cap (h(\ul{x},\cdot)\neq 0),$ as desired.  
\end{proof}

    \section{Proof of proposition \ref{prop:beta-bound}}

    The following result \cite[Proposition 4.4]{BDS} will play a key role in proving propositions \ref{prop:beta-bound} and \ref{prop:diag-to-good}. 

    \begin{lemma}\label{lem:good-subspace}
        Suppose that
        \[
        f(x,y,z,t_1,\ldots,t_m) = \alpha xy^{d-1}+\beta y^d+\gamma z^d+g(t_1,\ldots,t_m) 
        \]
        where $\alpha,\beta,\gamma\in K,$ $\alpha\gamma \neq 0$ and $g$ is some form. Then the $K$-points of $Z(f)$ are Zariski dense.   
    \end{lemma}

    \begin{definition}
        We say that $f$ as above is \emph{good}.
    \end{definition}

    Proposition \ref{prop:beta-bound} is a direct consequence of the following lemma.

    \begin{lemma}\label{lem:beta-recursion}
        Define a sequence by $n_0 = 2,n_1 = 4,$ and 
        $n_d = \frac{(\phi_d+11)n_{d-1}n_{d-2}}{2} $ for $d\ge 2.$ Then for all $d\ge 0,$ $n_d$ is an even integer and $\beta_d+1 \le n_d.$
    \end{lemma}

    \begin{proof}
        The statement about $n_d$ being an even integer is immediate. We prove the inequality by induction on $d.$ The case $d=0$ is immediate. For $d=1,$ we have $\beta_1\le 2$ since for any $d\ge 2, n>2$ and $a_1,\ldots,a_n \neq 0,$ the form $a_1x_1^d+\ldots+a_nx_n^d$ is irreducible.
        
        Suppose that that $\beta_{d'}+1 \le n_{d'}$ for $d' < d$ and let $\ul{f} = (f_1,\ldots,f_{d+1})$ be a nondegenerate system of diagonal forms with $n \ge  n_d.$ By lemma \ref{lem:good-subspace}, to prove the existence of a point $0\neq x\in K^n$ with $f_{d+1}(x) \neq 0$ and $f_i(x) = 0$ for $i = 1,\ldots,d,$ it's sufficient to construct a  $K$-subspace $L$ such that the following properties hold:
        \begin{enumerate}
            \item The forms $f_1,\ldots,f_{d-1}$ vanish on $L,$
            \item The restriction $f_d\restriction_L$ is good and
            \item We have $\codim_L Z(f_d,f_{d+1})\cap L = 2.$
        \end{enumerate}

        First, let $m = 3n_{d-1}$ and consider the subspace $K^m\cong L_2\subset K^n$ spanned by the final $m$ standard basis vectors. Write $L_2 = L_2^{(1)}\oplus L_2^{(2)}\oplus L_2^{(3)},$ where the $L_2^{(j)}$ are spanned by disjoint collections of $n_{d-1}$ standard basis vectors. By the inductive hypothesis, there exist $w^{(j)}\in L_2^{(j)} $ with $f_1(w^{(j)}) = \ldots = f_{d-1}(w^{(j)}) = 0$ and $f_{d+1}(w^{(j)}) \neq 0$ for $j=1,2,3.$ 

        Next,  let $m' =\frac{(\phi_d+1)n_{d-1}n_{d-2}}{2} + 2n_{d-1}$ and consider the subspace $K^{m'}\cong U\subset K^n$ spanned by the first $m'$ standard basis vectors. Note that $U\cap L_2 = \{0\}$ by the hypothesis on the size of $n.$ We write $U = U^{(1)}\oplus\ldots\oplus U^{(m'/n_{d-1})},$ where the $U^{(j)}$ are spanned by disjoint collections of $n_{d-1}$ standard basis vectors. By the inductive hypothesis, for each $j$ there is a $K$-point $u^{(j)}\in U^{(j)} $ with $f_1(u^{(j)}) = \ldots = f_{d-1}(u^{(j)}) = 0$ and $f_d(u^{(j)}) \neq 0.$ Thus far we have found linearly independent vectors $u^{(j)},w^{(j)}$ such that $\sp(u^{(j)},w^{(j)}) \subset Z(f_1,\ldots,f_{d-1}),$ $f_d(u^{(j)}) \neq 0$ and $f_{d+1}(w^{(j)})\neq 0.$ By construction, $u^{(j)},w^{(j)}$ are both $f_d$ and $f_{d+1}$ orthogonal.

        Now, let $m'' = \frac{(\phi_d+1)n_{d-2}}{2},$ and  $g(x_1,\ldots,x_{m''}) = f_d(x_1u^{(1)}+\ldots+x_{m''}u^{(m'')}).$ By construction, $g$ is a diagonal form with nonzero coefficients. Decompose $K^{m''} = V^{(1)}\oplus\ldots\oplus V^{(n_{d-2}/2)}$ where the $V^{(j)}$ are spanned by disjoint collections of $\phi_d+1$ standard basis vectors. For each $j$ there exists $0\neq x^{(j)}\in V^{(j)}$ with $g(x^{(j)}) = 0.$ Since $g$ has nonzero coefficients, each $x^{(j)}$ must have at least two nonzero entries. The vector $x = x^{(1)}+\ldots+x^{(n_{d-2}/2)}$ satisfies $g(x) = 0$ and has at least $n_{d-2}$ nonzero coefficients. For $1\le e\le d-1,$ consider the corresponding term of the Taylor expansion of $g(x+y),$
        \[
        g^{(d-e,e)}(x,y) =   \sum_{i=1}^{m''} f(u^{(j)}) \binom{d}{e} x_i^{d-e} y_i^e.
        \]
        This is a system of diagonal forms in $y$ with at least $n_{d-2}$ nonzero coefficients, so by the inductive hypothesis we can find $y\in K^{m''}$ with $g^{(d-e,e)}(x,y)  = 0$ for $1\le e< d-1$ and $g^{(1,d-1)} (x,y) \neq 0.$ This last condition implies in particular that $x,y$ are linearly independent. 

        Consider the subspace
        \[
        L = \sp(\sum_{i=1}^{m''}x_iu^{(i)}, \sum_{i=1}^{m''}y_iu^{(i)},u^{(m''+1)}, u^{(m''+2)}, w^{(1)},w^{(2)},w^{(3)}).
        \]
        By construction, the seven vectors spanning $L$ are linearly independent. We claim that $L$
        satisfies properties (1-3). Properties  (1) and (2) are straightforward to verify. For property (3), note that $\brk(f_d\restriction_L, f_{d+1}\restriction_L) \ge 3,$ so  $\codim_L (L\cap S(f_d,f_{d+1})) \ge 2$ by lemma \ref{lem:mat-diff-deg}. It follows that $f_d,f_{d+1}$ define a complete intersection in $L.$ 
    \end{proof}

    \begin{proof}[Proof of proposition \ref{prop:beta-bound}]
        
        We prove this by induction on $d,$ where the base cases $d=0,1$ are immediate. For $d\ge 2,$ applying lemma \ref{lem:beta-recursion} and invoking the inductive hypothesis yields
        \begin{align*}
            n_d &\le (\phi_d+11)2^{1+\varphi^{d-2}+\varphi^{d-3}} \prod_{k=2}^{d-1} (\phi_k+11)^{\varphi^{d-1-k}+\varphi^{d-2-k}} \\
            &= 2^{1+\varphi^{d-1}} \prod_{k=2}^d (\phi_k+11)^{\varphi^{d-k}}.
        \end{align*}
    \end{proof}

\section{Proof of proposition \ref{prop:diag-to-good}}

    The proof will follow a similar strategy to that of proposition \ref{prop:beta-bound}, but will require some additional steps to deal with the fact that our forms are not diagonal. We begin by introducing the auxiliary variety
    \[
    A_m = ((a_1,a_2,a_3,\ul{w},\ul{v}):(\ul{w},\ul{v}) \in D_m\ ,\ f(a_1w_1+a_2w_2+a_3w_3) = 0).
    \]
    
    \begin{lemma}\label{lem:A_m-irr}
        If $ \brk(\ul{f}) >  2\sum_{i=1}^d s_i \multiset{m}{i}+d'-1 $ then $A_m$ is irreducible.
    \end{lemma}

    \begin{proof}
         By lemma \ref{lem:brk-irr}, the varieties $Z_m(\ul{g})$ and $Z(\ul{g},f)$ are irreducible complete intersections. It follows that $f'(\ul{a},\ul{w}) = f(a_1w_1+a_2w_2+a_3w_3)$ doesn't vanish identically on $Z = \A^3\times D_m,$ so $A_m$ is a complete intersection. Let $S = S(f',T_m(\ul{g})).$ To prove the claim, it's sufficient  by lemma \ref{lem:R1-seq} to show that
         \begin{equation}\label{eq:A_m-sing}
             \dim Z\cap  S  < \dim A_m-1.
         \end{equation}
    
    To see this, decompose $ S = \bigcup_{a\in \A^3} S_a$ where $S_a = S\cap (\{a\}\times V^m).$  For $a = 0,$ we have
    \[
    \codim Z\cap  S_0 \ge \codim Z\cap (\{0\}\times V^m) = \codim Z+3   = \codim A_m+2.
    \]
    For any $a\neq 0,$ let  
    \[
    f_a(\ul{w},\ul{v}) = f(a_1w_1+a_2w_2+a_3w_3) = \sum_{e_1+e_2+e_3 = d} a_1^{e_1} a_2^{e_2} a_3^{e_3} f^{(e_1,e_2,e_3)}(\ul{w}).
    \]
    If $(a,\ul{w},\ul{v})\in S$ then $(\ul{w},\ul{v})\in S(f_a, T_m(\ul{g}))\subset S(T_m(\ul{f})).$ By lemma \ref{lem:brk-sing} and \ref{lem:brk-taylor} we get
    \begin{align*}
         \dim S_a &\le \dim(S(T_m(\ul{f}))) \le  m\dim V - \brk(\ul{f}) + \sum_{i=1}^d s_i\multiset{m}{i} + d' -2  \\
         &< m\dim V -  \sum_{i=1}^d s_i \multiset{m}{i} -1.
    \end{align*}
    Taking the union over $0\neq a\in A^3,$ and using the bound for $a=0,$ we get
    \[
    \dim Z\cap S <  m\dim V -  \sum_{i=1}^d s_i \multiset{m}{i} +2 \le \dim A_m-1,
    \]
    proving inequality \eqref{eq:A_m-sing}.
    \end{proof}

    The next lemma allows us to diagonalize while simultaneously keeping track of a Zariski open condition.

    \begin{lemma}
        Suppose that $\brk(\ul{f}) > 2\sum_{i=1}^d s_i \multiset{m}{i}+d'-1$ and that $h$ is a polynomial which doesn't vanish identically on $Z(\ul{f}).$ Then there exists a dense open subset of $(\ul{w},\ul{v})\in D_m$ such that 
        \begin{itemize}
        \item The vectors $w_1,w_2,w_3,v_1,\ldots,v_{m-3}$ are linearly independent,
        \item the polynomial $h$ doesn't vanish identically on $\sp(w_1,w_2,w_3)\cap (f=0)$ and 
        \item for each $i\in [m-3],$ $f(v_i) \neq 0.$ 
        \end{itemize}
    \end{lemma}

    \begin{proof}
         By lemma \ref{lem:brk-irr} the variety $D_m$ is irreducible so it's enough to show that each condition separately holds for a nonempty open subset.
         
         The first condition holds because the codimension of the variety of linearly dependent tuples in $V^m$ is
        \[
        \dim V - (m-1) > \brk(\ul{f}) - m > \sum_{i=1}^d s_i \multiset{m}{i} > \codim_{V^m} D_m. 
        \]
      
      For the second condition, note that our hypothesis implies that 
      $$A' = A_m \cap (h(a_1w_1+a_2w_2+a_3w_3) = 0)$$ is a proper subset of $A_m.$ By lemma \ref{lem:A_m-irr} this implies $\dim A' < \dim A_m,$  and therefore the fiber satisfies  $\dim A'(\ul{w},\ul{v}) < \dim A_m(\ul{w},\ul{v})$ for a dense open subset of $D_m.$ 
      
    For the third condition, it is again enough to prove that it holds on a dense open subset for each fixed $i.$ Indeed, by lemma \ref{lem:brk-irr}, for each $i$ the set $D_m\cap (f(v_i) =0) \subset D_m$ has lower dimension.
    \end{proof}

    \begin{proof}[Proof of proposition \ref{prop:diag-to-good}]

    Suppose $\brk(\ul{f}) > \max(\delta_m(s_d,\ldots,s_1), 2\sum_{i=1}^d\multiset{m}{i}+d'-1) $ and let $h$ be a polynomial which doesn't vanish identically on $Z(\ul{f}).$ We wish to show that $Z(\ul{f})\cap (h\neq 0)$ contains a $K$-point. By lemma \ref{lem:good-subspace}, it's enough to find a $K$-subspace $L\subset K^n$ such that: 
    \begin{enumerate}
        \item The forms in $\ul{g}$ vanish on $L,$
        \item The restriction $f\restriction_L$ is good and
        \item the polynomial $h$ doesn't vanish identically on $Z(f)\cap L.$
    \end{enumerate}
    
    Let $U\subset D_m$ be the dense open subset of the previous lemma and let  $(\ul{w},\ul{v})\in U$ be a $K$-point. This exists because $\brk(\ul{f}) > \delta_m.$ Let 
    $$ f_0(x_1,\ldots,x_{m-4}) = f(x_1v_1+\ldots+x_{m-4}v_{m-4})$$
    be the nondegenerate diagonal form given by the restriction of $f$ and decompose $K^{m-4} = V^{(1)}\oplus\ldots\oplus V^{(\lceil\frac{\beta_{d-1}+1}{2}\rceil)},$ where the $V^{(j)}$ are spanned by disjoint collections of $\phi_d+1$ standard basis vectors. By definition, for each $j$ there exists $0\neq x^{(j)}\in V^{(j)}$ with $f_0(x^{(j)}) = 0.$ Each $x^{(j)}$ must have at least two nonzero entries, so the vector $x = x^{(1)}+\ldots+x^{\left( \lceil\frac{\beta_{d-1}+1)}{2}\rceil \right)}$ satisfies $f_0(x) = 0$ and has at least $\beta_{d-1}+1$ nonzero entries.

    For $1\le e \le d-1,$ consider the terms of the Taylor expansion of $f_0(x+y),$
        \[
        f_0^{(d-e,e)}(x,y) =   \sum_{i=1}^{m-4} f(v_j) \binom{d}{e}  x_i^{d-e} y_i^e.
        \]
    This is a system of diagonal forms in $y$ with at least $\beta_{d-1}+1$ nonzero coefficients, so we can find $y\in K^{m-4}$ with $f_0^{(d-e,e)}(x,y)  = 0$ for $1\le e< d-1$ and $f_0^{(1,d-1)} (x,y) \neq 0.$ This last condition implies in particular that $x,y$ are linearly independent. By construction, the subspace
    \[
    L = \sp\left(\sum_{i=1}^{m-4} x_iv_i, \sum_{i=1}^{m-4} y_iv_i ,  v_{m-3}, w_1,w_2,w_3 \right) 
    \]
    satisfies conditions (1-3) above. This completes the proof.
\end{proof}

\bibliographystyle{plain}
\bibliography{refs}

\end{document}